\documentclass[preprint,3p,times]{elsarticle}
\usepackage{amsmath,amssymb}
\usepackage[all]{xy}
\usepackage{latexsym}
\usepackage{amsthm,a4wide,color}
\usepackage{amsmath,amscd,verbatim}
\usepackage{hyperref}
\usepackage{mathrsfs}
\usepackage{graphicx}

\input diagxy
\newcommand\thda{\mathrel{\rotatebox[origin=c]{-90}{$\twoheadrightarrow$}}}
\newcommand\thua{\mathrel{\rotatebox[origin=c]{90}{$\right) $}}}
\theoremstyle{plain}
\newtheorem{thm}{Theorem}[section]
\newtheorem{lem}[thm]{Lemma}
\newtheorem{prop}[thm]{Proposition}
\newtheorem{cor}[thm]{Corollary}
\newtheorem{fa}[thm]{Fact}
\theoremstyle{definition}
\newtheorem{defn}[thm]{Definition}
\newtheorem{exmp}[thm]{Example}

\begin{document}
\newcommand{\Ran}{\mathrm{Ran}}
\newcommand{\id}{\mathrm{id}}
\newcommand{\p}{{\preccurlyeq}}
\newcommand{\pp}{{\preceq}}
\newcommand{\sq}{{\sqsubseteq}}
\newcommand{\va}{{\vartriangle}}
\newcommand{\ua}{{\uparrow}}
\newcommand{\da}{{\downarrow}}
\newcommand{\la}{{\lambda}}
\newcommand{\al}{{\alpha}}
\newcommand{\ra}{{\rightarrow}}

\numberwithin{equation}{section}

\renewcommand{\theequation}{\thesection.\arabic{equation}}

\begin{frontmatter}

\title{ When is the convolution  a t-norm on normal, convex and  upper semicontinuous fuzzy truth values?}

\author{Jie Sun\corref{cor}}
\ead{jiesun1027@163.com}

\cortext[cor]{Corresponding author.}
\address{School of Mathematics,  Southwest Minzu University, Chengdu 610041, China}

\begin{abstract}
In Type-2 rule-based  fuzzy systems (T2 RFSs), triangular norms on complete lattice  \((\mathbf{L},\sq)\)  or \((\mathbf{L_u},\sq)\) can be used to model the compositional rule of inference, where  \(\textbf{L}\) is the set of all convex normal fuzzy truth values,  \(\mathbf{L_u}\) is the set of all convex normal and upper semicontinuous fuzzy truth values,  and \(\sq\) is the so-called  convolution order. Hence, the  choice of t-norms  on \((\mathbf{L},\sq)\) or \((\mathbf{L_u},\sq)\) may influence the performance  of T2 RFSs, and thus, it is significant to broad the set of t-norms among which domain  experts can choose most suitable one. To construct t-norms on \((\mathbf{L},\sq)\) or \((\mathbf{L_u},\sq)\),  the mainstream method is based on convolution  \(\ast_\va\)  induced by two operators \(\ast\) and \(\va\) on the unit interval \([0,1]\).  Recently, we have  complete solve the question when   convolution  \(\ast_\va\) is a t-norm on \((\mathbf{L},\sq)\).  This paper  aim to provide the  necessary and sufficient conditions under which convolution  \(\ast_\va\)  is a t-norm on \((\mathbf{L_u}, \sq)\).\end{abstract}

\end{frontmatter}

\section{Introduction}

Triangular norms, or t-norms for short, were introduced by Schweizer and Sklar to generalize the triangle inequality from metric spaces to probabilistic metric spaces \cite{s1}. Later on,  t-norms are often used for modeling the intersection of fuzzy sets \cite{a} and   the conjunction in fuzzy logics \cite{h}. Furthermore, the applications of t-norms in  generalized measures and integrals, fuzzy control and many other areas can be found in the literature, for example, a well known one is \cite{Klement2000}. 

As the generalization of type-1 fuzzy sets,  type-2 fuzzy sets were introduced by Zadeh in 1975 \cite{za75}.  In  type-2 fuzzy set,  the  membership degree of a primary element is usually a fuzzy truth value (the  map from \([0,1]\) to \([0,1]\)).  Since type-2 fuzzy sets can better model the  fuzziness in real life, and were widely used in approximation \cite{me05}, aggregation \cite{rh01}, control \cite{ca05}, database \cite{ni06}, ect in recent decades.

In T2 RFS, t-norms on complete lattice  \((\mathbf{L},\sq)\)  or \((\mathbf{L_u},\sq)\) can be used to model the compositional rule of inference \cite{me24,su25-11,su25}, where \(\mathbf{L}\) is the set of all  normal convex fuzzy truth values,  \(\mathbf{L_u}\) is the set of all  normal convex and upper semicontinuous fuzzy truth values, and \(\sq\) is the so-called convolution order (or extended order). Hence, the  choice of t-norms  on \((\mathbf{L},\sq)\) or \((\mathbf{L_u},\sq)\) may influence the performance  of T2 RFS. Therefore, it is significant to broad the set of t-norms among which domain  experts can choose most suitable one. 

Convolution operator  has a close relation with  Zadeh's extension principle \cite{za75}, and have been widely studied as   (generalized) extended t-norm  \cite{he22,wu20,zh19}, uninorm \cite{li22,xi18}, implication \cite{wa15,wa18}, aggregation \cite{ta14,to17}, overlap function \cite{ji22,li20}, etc.

In terms of the construction of t-norms on  \((\mathbf{L},\sq)\)  or \((\mathbf{L_u},\sq)\),  a key problem is when convolution operator is a t-norm. In paper \cite{su25-11}, we have given  the necessary and sufficient conditions such that the convolution is a t-norm on  \((\mathbf{L},\sq)\).   However, there is few results about the conditions such that convolution is a t-norm on \((\mathbf{L_u},\sq)\).  In this paper, we aim to give the necessary and sufficient conditions such that the convolution is a t-norm on \((\mathbf{L_u},\sq)\).

 The paper are organized as follows. Section 2 reviews  definitions and properties of t-norms, \(t_r\)-norms,  fuzzy truth values, and convolutions. Section 3 is devoted to the proof of our main results.

\section{Preliminaries}

\subsection{Triangular norms}

Let  \((P,\le, 0_p, 1_p)\) be a bounded poset  with minimum \(0_p\) and maximum \(1_p\). A binary function \(\ast:P^2\to P\) is  called a t-norm, if the following four axioms are satisfied:
\begin{itemize}
\item[(T1)]     \(x\ast y = y\ast x\);
\item[(T2)]     \((x\ast y)\ast z=x\ast (y\ast z)\);
\item[(T3)]     \(x\ast y\le  x\ast z\) whenever \(y\le z\);
\item[(T4)]   \(x\ast 1_p=x\).
\end{itemize}
\begin{exmp}
The product t-norm \(T_p:[0,1]^2\to[0,1]\)  and  the {\L}uckasiewicz  t-norm \(T_{L}:[0,1]^2\to[0,1]\),  are respectively  given by \[T_p(x,y)=x \cdot y,\quad  T_{L}(x,y)=\max\{x+y-1,0\}.\]  
\end{exmp}
 A t-norm \(\ast:[a,b]^2\to [a,b]\) is called to be 
 \begin{itemize}
\item \emph{continuous}, if it is continuous as a binary function. 
\item \emph{left continuous}, if for any monotone  increasing sequences \(\{x_n\}_{n\in \mathbb{N}}\) and \(\{y_n\}_{n\in  \mathbb{N}}\), \[\lim  x_n\ast y_n= \lim x_n \ast  \lim y_n.\]
\item \emph{right continuous}, if for any monotone  decreasing sequences \(\{x_n\}_{n\in \mathbb{N}}\) and \(\{y_n\}_{n\in  \mathbb{N}}\), \[\lim  x_n\ast y_n= \lim x_n \ast  \lim y_n.\]
\item  \emph{conditionally cancellative},  if  \(x_1\ast y= x_2\ast y> a\) implies  \(x_1=x_2\).  

 \end{itemize}
  The following proposition is obtained from the monograph \cite{al06}.

 \begin{prop} 
 A t-norm \(\ast:[a,b]^2\to[a,b]\) is  continuous (left continuous, right continuous)  if and only if   for each \(x\in[a,b]\), the section \(x\ast(-):[a,b]\to [a,b]\) is continuous (left continuous, right continuous).
\end{prop}
 \subsection{Fuzzy truth values}

A \emph{fuzzy truth value} is a map from the unit interval \([0,1]\) to itself.
Let \(\mathbf{M}\) be the set of all fuzzy truth values.  A \emph{type-2 fuzzy set} \(A\) is determined  by a  membership function \[\mu_A:  X\to \mathbf{M}.\]

 For each  \(\alpha\in [0,1]\), the \(\alpha\)-cut and  strong \(\alpha\)-cut of  \(f\in \mathbf{M}\) are respectively given by  
\[f^{\alpha}=\{x\in[0,1]\mid f(x)\ge \alpha\},\]
\[f^{\hat{\alpha}}=\{x\in[0,1]\mid f(x)> \alpha\}.\]

\begin{defn}
 A fuzzy truth value \( f\in \textbf{M} \) is called to be  
 \begin{itemize}
 \item \emph{normal}, if \(\sup\{f(x)\mid x\in[0,1]\}=1\).
 \item \textit{convex}, if \(x\le y\le z\) implies \(f(x)\wedge f(z)\le f(y)\).
 \item \textit{upper semicontinuous},  if  \(f^\alpha\) is closed for all \(\alpha\in[0,1]\). 
 \end{itemize}
\end{defn}

There exist some equivalent statements of the concept  \textit{convex}.  (see e.g. \cite{wa05}).

\begin{prop}\label{10}For \(f\in \mathbf{M}\), the following statements are equivalent.
\begin{itemize}
\item[(i)] \(f\) is convex.
\item[(ii)]  \(f^{\alpha}\) is  convex  for each \(\alpha\in [0,1]\).
\item[(iii)]  \(f^{\hat{\alpha}}\) is  convex  for each \(\alpha\in [0,1]\).
 \item[(iv)] there exists a  monotone increasing map \(f_1\) and a  monotone decreasing map \(f_2\) such that   \(f(x)=f_1(x)\wedge f_2(x).\)
\end{itemize}
\end{prop}

Let \(\mathbf{L}\) denote the set of all normal convex fuzzy truth values, and 
let \(\mathbf{L_u}\) denote the set of all normal convex  and upper semicontinuous   fuzzy truth values.

\begin{lem}\label{1}
Suppose that \(f\in \mathbf{L_u}\) and  the sequence \(\{a_i\}\) is monotone with limit \(a_0\).  Then
\begin{itemize}
\item[(i)] If \(\{f(a_i)\}\) is monotone increasing, then \(\sup f(a_i)\le f(a_0)\).
\item[(i)] If \(\{f(a_i)\}\) is monotone decreasing, then \(\inf f(a_i)=f(a_0)\).
\end{itemize}
\end{lem}
\begin{proof}
If not, it is easy to check that \(f\) is not upper semicontinuous.
\end{proof}

\begin{prop}\label{2.5} 
If   \(f\in \mathbf{L_u}\), then \(f\) can attains value 1.
\end{prop}
\begin{proof}  For each \(i\in \mathbb{N}^+\),   let  \[a_i=\min \{x\in [0,1]\mid f(x)\ge 1-\frac{1}{i}\}.\] Obviously, both \(\{a_i\}\) and \(\{f(a_i)\}\) are monotone increasing. Suppose \(\sup \{a_i\}=a_0\). By Lemma \ref{1}, \[f(a_0)\ge \sup f(a_i)=1,\]  hence, \(f(a_0)=1\).
\end{proof}
\subsection{Convolutions}

\begin{defn}\label{de} \label{df} Let   \(\ast:[0,1]^2\to[0,1]\) and \(\va:[0,1]^2\to[0,1]\) be two binary  operations on \([0,1]\). The \emph{convolution} induced by  \(\ast\) and  \(\va\) is a binary operation \(\ast_\va:\mathbf{M}^2\to \mathbf{M}\)  given by \[(f \ast_\va g)(x)=\bigvee_{a\ast b=x} f(a) \va g(b).\label{1.1}\] In particular,  if there not exist \((a,b)\in [0,1]^2\)
 such that \(a\ast b=x\),  then \[(f \ast_\va g)(x)=\bigvee \varnothing=0.\] \end{defn}

It is easy to see that for any convolution \(\ast_\va\),  maps \(f,g\in \mathbf{M}\) and numbers \(a,b\in [0,1]\),
\begin{equation}\label{2.1} (f\ast_\va g)(a\ast b)=\bigvee_{x\ast y=a\ast b}f(x)\va g(y)\ge f(a)\va g(b).\end{equation}

\begin{prop}\cite{su25-11}\label{.14}
For any convolution \(\ast_\va\),  if  \(\va\) is monotone increasing in each place, then for each \( f,g\in \mathbf{M}\) and \(\alpha\in [0,1]\), \[ (f\ast_\va g)^{\hat{\alpha}}=\bigcup_{\alpha_1\va \alpha_2>\alpha} f^{\alpha_1}\ast g^{\alpha_2}.\]
\end{prop}

\begin{defn} The  partial orders \(\sq\)  on \(\mathbf{M}\) defined  by \[f\sqsubseteq g, \text{ if  } f\wedge_{\wedge} g=f\]
is called  to be \emph{convolution order}.
\end{defn}
It is known that \((\mathbf{M}, \sqsubseteq)\) is an inf semilattice \cite{wa05}, and the bounded posets \((\mathbf{L},\sqsubseteq,\bar{0},\bar{1})\),  \((\mathbf{L_u},\sqsubseteq,\bar{0},\bar{1})\)  are  a complete lattice and a  completely distributive lattice, respectively (see \cite{ha08}).

\subsection{\(t_r\)-norms}

For any \(A\subseteq[0,1]\), let \(\bar{A}:[0,1]\to[0,1]\) denote the \emph{characteristic function} of \(A\), that is \[\bar{A}(x)=\begin{cases}1,\quad  \text{if } x\in A,\\ 0, \quad \text{otherwise}.\end{cases}\]  Put  \[\mathbf{J}=\{\bar{x}\mid x\in[0,1]\},\]  
and  \[\mathbf{J}^{[2]}=\{\overline{[a,b]}\mid 0\le a \le b \le 1\}.\]

Similar to the definition of  \(t_r\)-norm on \((\mathbf{L},\sq)\) given in \cite{he14}, we give the definition of  \(t_r\)-norm on \((\mathbf{L_u},\sq)\) as follows.
\begin{defn}
A binary operation \(T : \mathbf{L_u}^2\to \mathbf{L_u}\) is called a
\(t_r\)-norm on \((\mathbf{L_u},\sq)\),  if \(T\) is a t-norm on \((\mathbf{L_u},\sq)\), \(T\) is closed on \(\mathbf{J}\) and \(\mathbf{J}^{[2]},\) respectively, and \(T(\overline{[0, 1]}, \overline{[a,b]})=\overline{[0,b]}.\) 
\end{defn}

\begin{prop}\cite{su25-11}\label{8} If \(\ast_\va\) is  a   t-norm on \((\mathbf{L}_\mathbf{u},\sq)\),  then
\begin{itemize}
\item[(i)] \(\ast\)  is a  continuous t-norm and \(\va\) is a t-norm.
\item[(ii)] \(\ast_\va\) is  closed on  \(\mathbf{J}\) and \(\mathbf{J}^{[2]}\), respectively.
\item[(iii)]  \(\overline{[0, 1]}\ast_\va \overline{[a,b]}=\overline{[0,b]}.\)
\end{itemize}
\end{prop}

By the above proposition, we have the following result immediately. 
\begin{prop}\label{0.1}
For any  two binary functions \(\ast\) and \(\va\)  on the unit interval \([0,1]\), the convolution \(\ast_\va\) is  a   t-norm on \((\mathbf{L}_\mathbf{u},\sq)\) if and only if it is a \(t_r\)-norm on  \((\mathbf{L}_\mathbf{u},\sq)\).
\end{prop}

\section{Main  Results}

In this section, we provide a complete answer to the
 question when convolution is a t-norm on \((\mathbf{L_u}, \sq)\)  as follows.
 \begin{thm}\label{0}
For any  two binary functions \(\ast\) and \(\va\)  on the unit interval \([0,1]\), the convolution \(\ast_\va\) is a t-norm on \((\mathbf{L_u}, \sq)\)  if and only if \(\ast\) is a continuous t-norm  and \(\va\) is a right continuous t-norm.
 \end{thm}

\subsection{Sufficiency}

In the subsection, we aim to  prove the sufficiency of  Theorem \ref{0}.

\begin{fa}\label{9}
Suppose that \(\{A_i\}_{i\in  \Lambda}\) is a family of  convex subsets of  unit interval \([0,1]\).  If  \(\bigcap_{i\in  \Lambda} A_i\ne \varnothing\),  then   \(\bigcup_{i\in  \Lambda} A_i\) is convex.
\end{fa} 

\begin{prop}\label{2.8}
If   \(\ast\) is a continuous t-norm and  \(\va\) is a right continuous t-norm, then for any  \(f,g\in \mathbf{L_u}\),   \(f\ast_\va g\) is convex  and  normal.
\end{prop}
\begin{proof}
Since \(f,g\in \mathbf{L_u}\),  by Proposition \ref{2.5}, there exist \(a_0, b_0 \in[0,1]\) such that  \(f(a_0)=g(b_0)=1\).
By (\ref{2.1}), \[(f\ast_\va g)(a_0\ast b_0)\ge f(a_0)\va g(b_0)=1.\] Hence, \(f\ast_\va g\) is normal. 


For each \(\al\in[0,1)\),  \[\bigcap_{\alpha_1\va \alpha_2>\alpha} f^{\alpha_1}\ast g^{\alpha_2}\supseteq   f^1\ast g^{1}  \ne \varnothing\] By Propostion \ref{.14} and Fact \ref{9},    \[(f\ast_\va g)^{\hat{\alpha}}=\bigcup_{\alpha_1\va \alpha_2>\alpha} f^{\alpha_1}\ast g^{\alpha_2}\] is convex. 
Hence, \(f\ast_\va g\) is convex. \end{proof}

 The following lemma  is the key to prove   Proposition \ref{2.10}.
\begin{lem} \label{2.9}
Suppose that  \(\ast\) is a continuous t-norm and  \(\va\) is a right continuous t-norm, \(f,g\in \mathbf{L_u}\), and \(\al \in[0,1]\). Then  for any \(x\in\overline{ (f\ast_{\va} g)^\al}\), there exists \((a, b)\in [0,1]^2\) such that \(a\ast b=x\) and \(f(a)\va g(b)\ge \al\). 
\end{lem}
\begin{proof}
Suppose  that  \(x\in (f\ast_\va g)^\al\). If \((f\ast_\va g)^\al\) is a singelton, then \((f\ast_\va g)^\al=\{x\}\). Since \(f,g\in \mathbf{L_u}\), then there exists \(a,b\in [0,1]\) such that  \(f(a)=g(b)=1\). By (\ref{2.1}),  \[(f\ast_\va g)(a\ast b)\ge f(a)\va g(b)= 1\ge \alpha,\]  i.e.  \[a\ast b\in (f\ast_{\va} g)^\al,\] hence, \(a\ast b=x\). Besides, \(f(a)\va g(b)=1\ge \al\). 
 
If \((f\ast_\va g)^{\al}\)  is not a singleton,   we shall show that  there exists \(a, b\in[0,1]\) such that \(a\ast b=x\) and \(f(a)\va g(b)\ge \al\) by 3 steps.



Step 1,  determine  the numbers  \(a, b\in[0,1]\) as  follows. Since the convex set \((f\ast_\va g)^{\al}\)  is not a singleton and \(x\in \overline{(f\ast_{\va} g)^\al}\),   there exists a strictly monotone sequence \(\{x_n\}_{ n\in \mathbb{N}^+}\) in \((f\ast_\va g)^{\al}\) such that \[\lim_{n\rightarrow \infty}x_n=x.\]   Suppose that \(\{x_n\}_{ n\in \mathbb{N}^+}\) is strictly decreasing ( if \(\{x_n\}_{ n\in \mathbb{N}^+}\) is strictly increasing, the proof is  similar), then \[\inf x_n=x.\] For each \(n\in \mathbb{N}^+\), since  \(x_n\in (f\ast_\va g)^{\al}\), i.e. \[(f\ast_\va g)(x_n)\ge \al,\] then there exists \((a_n, b_n)\in [0,1]^2\) such that \begin{equation}\label{q1}a_n\ast b_n= x_n\end{equation} and \begin{equation}\label{q2} f(a_n)\va g(b_n)\ge \al-\frac{1}{n}.\end{equation} Since \(\{x_n\}_{ n\in \mathbb{N}^+}\) is strictly decreasing,   (\ref{q1}) implies that \((a_n, b_n)\ne(a_m,b_m)\) whenever \(n\ne m,\) thus the set \(\{(a_n,b_n)\mid n\in \mathbb{N}^+\}\) is infinite. By compactness, there exists an accumulation point \((a, b)\) of  set \(\{(a_n,b_n)\mid n\in \mathbb{N}^+\}\).

Step 2,  prove  \(a\ast b =x\).   Since for each \(n\in \mathbb{N}^+\), \[a_n\ast b_n=x_n> x,\] then \[\{(a_n,b_n)\mid n\in \mathbb{N}^+\}\subseteq \{(x,y)\mid x\ast y\ge x\}.\]  Hence,  the accumulation point \[(a,b)\in \{(x,y)\mid x\ast y\ge x\} ,\] thus, \(a\ast b \ge x\). If \(a\ast b \ne x\), let  \(x'\in(x, a\ast b)\). Since \(\{x_n\}_{ n\in \mathbb{N}^+}\) is strictly decreasing and  \[\lim_{n\rightarrow \infty}x_n=x<x',\]  then there exists \(N\), for any \(n> N\),   \(x_n<x'\),  then  \(a_n\ast b_n=x_n< x'\) from  (\ref{q1}), thus  \[(a_n,b_n)\in \{(x,y)\mid x\ast y\le x'\}.\] Hence, the accumulation point\[(a, b)\in \{(x,y)\mid x\ast y\le x'\},\]
 and thus,  \(a\ast b\le x'\), which is contradict to  \(x'\in(x, a\ast b)\). Therefore,  \(a\ast b =x.\)   
 
 Trivially, there exists a subsequence \(\{(a_{n_i},b_{n_i})\}_{i\in \mathbb{N}^+}\) of \(\{(a_n,b_n)\}_{n\in \mathbb{N}^+}\) satisfying all the following conditions 
\begin{itemize}
\item both \(\{a_{n_i}\}_{i\in \mathbb{N}^+}\) and \(\{b_{n_i}\}_{i\in \mathbb{N}^+}\) are monotone and converges to \(a\) and \(b\), respectively.  
\item both \(\{f(a_{n_i})\}_{i\in \mathbb{N}^+}\) and \(\{g(b_{n_i})\}_{i\in \mathbb{N}^+}\) are monotone.
\end{itemize}

 Step 3, prove that \(f(a)\va g(b)\ge \al\) case by case.
 
 
 Case 1,  both \(\{f(a_{n_i})\}_{i\in \mathbb{N}^+}\) and \(\{g(b_{n_i})\}_{i\in \mathbb{N}^+}\)  are monotone increasing. By lemma \ref{1}, \(\sup f(a_{n_i})\le f(a)\) and  \(\sup g(b_{n_i})\le g(b)\). Hence,  \[f(a)\va g(b)\ge \sup f(a_{n_i}) \va \sup g(b_{n_i})\ge \sup f(a_{n_i}) \va  g(b_{n_i}).\] This, together with (\ref{q2}), implies  \[f(a)\va g(b)\ge \al.\]

Case 2,  both \(\{f(a_{n_i})\}_{i\in \mathbb{N}^+}\) and \(\{g(b_{n_i})\}_{i\in \mathbb{N}^+}\) are monotone  decreasing. By Lemma \(\ref{1}\),  \[\inf f(a_{n_i})=f(a)\quad \inf g(b_{n_i})=g(b).\] Since \(\va\) is right continuous,  then \[\lim_{i\ra\infty}f(a_{n_i})\va g(b_{n_i})=f(a)\va g(b).\] On the other hand, by (\ref{q2}), \[ \lim_{i\ra\infty}f(a_{n_i})\va g(b_{n_i})\ge \al.\] Hence, \[f(a)\va g(b)\ge \al.\]

Case 3,   without loss of generality, suppose that  \(\{f(a_{n_i})\}_{i\in \mathbb{N}^+}\) is monotone increasing and \(\{g(b_{n_i})\}_{i\in \mathbb{N}^+}\) is monotone decreasing. 
Firstly,  we show that for each \(k\in \mathbb{N}^+\), \begin{equation}\label{3}(\sup f(a_{n_i})) \va g(b_{n_k})\ge \al.\end{equation} In fact, for any \(j\in \mathbb{N}^+\)  satisfying \(j>k\), since \(\{g(b_{n_i})\}_{i\in \mathbb{N}^+}\) is monotone decreasing,  \(g(b_{n_j})\le g(b_{n_k})\), then \[ (\sup f(a_{n_i}))\va g(b_{n_k})\ge  (\sup f(a_{n_i}))\va g(b_{n_j})\ge  f(a_{n_j})\va g(b_{n_j}).\]  This, together with  (\ref{q2}), implies that (\ref{3}) holds.
Next,  by Lemma \ref{1}, \[f(a)\ge \sup f(a_{n_i}), \quad g(b)=\inf g(b_{n_k}).\] 
thus, \[f(a)\va  g(b)\ge (\sup f(a_{n_i}))\va (\inf g(b_{n_k})).\]
Since \(\va\) is right continuous, \[ (\sup f(a_{n_i}))\va (\inf g(b_{n_k}))=\inf ((\sup f(a_{n_i}))\va g(b_{n_k})).\] By (\ref{3}), 
\[\inf ((\sup f(a_{n_i}))\va g(b_{n_k}))\ge \al.\] Therefore, \(f(a)\va  g(b)\ge \alpha\).
\end{proof}

\begin{prop}\label{2.10}
If   \(\ast\) is a continuous t-norm and  \(\va\) is a right continuous t-norm, then for any  \(f,g\in \mathbf{L_u}\),   \(f\ast_\va g\) is upper semicontionuous.
\end{prop}
\begin{proof}
For any \(\alpha \in [0,1]\),  if  \((f\ast_\va g)^{\al}\) is a singleton, then it is closed.
If  \((f\ast_\va g)^{\al}\) is not a singleton, By Lemma \ref{2.9}, for any \(x\in \overline{(f\ast_{\va} g)^\al}\),  there exists \((a, b)\) such that \(a\ast b=x\) and \(f(a)\va g(b)\ge \al\). By (\ref{2.1}), \[(f\ast_{\va} g)(x)=(f\ast_{\va} g)(a\ast b)\ge  f(a)\va  g(b)\ge \al,\]  i.e.  \(x\in(f\ast_{\va} g)^\al.\) Hence, \((f\ast_{\va} g)^\al\) is closed,  we complete the proof. 
\end{proof}

\begin{prop}\label{3.5}
If   \(\ast\) is a continuous t-norm and  \(\va\) is a right continuous t-norm, then convolution \(\ast_\va\) is  a binary function  on \( \mathbf{L_u}\).
\end{prop}
\begin{proof}
It is result of   Proposition \ref{2.8} and Proposition \ref{2.10}.
\end{proof}

\begin{lem}\label{2.11}
If   \(\ast\) is a continuous t-norm and  \(\va\) is a right continuous t-norm, then for any \(f,g\in \mathbf{L_u}, x\in[0, 1]\),  there exists \(a,b\in[0,1]\) such that \(a\ast b=x\) and \((f\ast_{\va}g)(x) =f(a)\va f(b)\).
\end{lem}
\begin{proof}
 For any \(f,g\in \mathbf{L_u}, x\in[0, 1]\), let \(\al=(f\ast_\va g)(x)\), then \(x\in (f\ast_\va g)^\al.\)   By Proposition \ref{2.10},  \(f\ast_\va g\) is upper semicontinuous, thus  \((f\ast_\va g)^\al=\overline{(f\ast_\va g)^\al}\). Since \(x\in \overline{(f\ast_\va g)^\al},\)  by Lemma \ref{2.9}, there exists \((a,b)\in [0,1]^2\) such that   \(a\ast b=x\) and \(f(a)\va g(b)\ge \al\). On the other hand, by (\ref{2.1}), \[f(a)\va g(b)\le (f\ast_\va g)(a\ast b)=(f\ast_\va g)(x)=\al.\] Therefore, \(f(a)\va g(b)=\al=(f\ast_\va g)(x)\).
\end{proof}

\begin{cor} If  \(\ast\) is a continuous t-norm and  \(\va\) is a right continuous t-norm, then for any \(f,g\in \mathbf{L_u}\) and \(x\in[0, 1]\),   \((f\ast_\va g)(x)=\max\{f(a)\va g(b)\mid  a\ast b=x\}\). \end{cor}

\begin{proof}
It follows from Lemma \ref{2.11} immediately.
\end{proof}

\begin{prop}\label{2.13}
If   \(\ast\) is a continuous t-norm and  \(\va\) is a right continuous t-norm,  then for any \(f,g,h\in \mathbf{L_u}\) and \(x\in[0,1]\), 

\[((f\ast_\va g)\ast_\va h)(x)=\bigvee_{x_1\ast x_2\ast x_3=x} f(x_1)\va g(x_2)\va h(x_3)=(f\ast_\va (g\ast_\va h))(x).\]
Hence, \(\ast_\va\) is associative on \(\mathbf{L_u}\).
\end{prop}
\begin{proof}
 We only show that \(((f\ast_\va g)\ast_\va h)(x)=\bigvee_{x_1\ast x_2\ast x_3=x} f(x_1)\va g(x_2)\va h(x_3)\),  and another equality is similar.
On the one hand, by Lemma \ref{2.11}, there exist  \(a,b\in[0,1]\), such that \(a\ast b=x\) and \(((f\ast_\va g)\ast_\va h)(x)=(f\ast_\va g)(a)\va h(b)\). There also exist \(a_1,a_2\in [0,1]\)  such that \(a_1\ast a_2=a\) and \((f\ast_\va g)(a)=f(a_1)\va g(a_2)\). Hence, \[((f\ast_\va g)\ast_\va h)(x)=f(a_1)\va g(a_2)\va h(b).\] Since \(a_1\ast a_2\ast b=x\), \[((f\ast_\va g)\ast_\va h)(x)\le \bigvee_{x_1\ast x_2\ast x_3=x} f(x_1)\va g(x_2)\va h(x_3).\]

On the other hand,\[((f\ast_\va g)\ast_\va h)(x)=\bigvee_{a\ast b=x} (f\ast_\va g)(a)\va h(b)=\bigvee_{a\ast b=x}( \bigvee_{a_1\ast a_2=a} f(a_1)\va g(a_2) ) \va h(b)\ge\bigvee_{a_1\ast a_2\ast b=x} f(a_1)\va g(a_2)\va h(b).\] Therefore,  \[((f\ast_\va g)\ast_\va h)(x)=\bigvee_{x_1\ast x_2\ast x_3=x} f(x_1)\va g(x_2)\va h(x_3).\] 
\end{proof}

A non-empty  convex subset of  \([0,1]\) is called an \emph{subinterval} of \([0,1]\). Let \(\mathcal{I}\) denote the set of all subinterval of \([0,1]\).  
Define the partial order  \(\p\) on \(\mathcal{I}\)  as follows, 
   \[A\p B,\text { if } A\wedge B=A,\]  
 where \(A\wedge B=\{x\wedge y\mid x\in A, y\in B\}\).

\begin{prop}\cite{su25-11}\label{17}
For all \(f,g\in \mathbf{L}\), 
 the following statements are equivalent.
\begin{itemize}
\item[(i)] \(f \sqsubseteq g.\)
\item[(ii)]  \(f^{\hat{a}}\p g^{\hat{a}}\)  for all  \(a\in(0,1)\).
\item[(iii)] \(f^{a}\p g^{a}\)  for all  \(a\in(0,1)\).
\end{itemize}
\end{prop}

\begin{lem}\cite{su25-11} \label{18} If   \(\ast\)  is a continuous t-norm, then for \(A, B,C\in \mathcal{I}\),  \(A\p B\) implies  \(A\ast C\p B\ast C\).
\end{lem}

\begin{lem}\cite{su25-11}\label{15}
Suppose that  \(\{A_i\}_{i\in \Lambda}\) and   \(\{B_i\}_{i\in \Lambda}\)  are two families of  subintervals of \([0,1]\), and  \(A_i\p B_i\) for each \(i\in \Lambda\).  If both \(\bigcup_{i\in \Lambda}{A_i}\) and \(\bigcup_{i\in \Lambda}{B_i}\) are convex, then \[\bigcup_{i\in \Lambda} A_i \p  \bigcup_{i\in \Lambda} B_i\]
 If  both \(\bigcap_{i\in \Lambda}A_i\) and  \(\bigcap_{i\in \Lambda}B_i\) are non-empty, then  \[\bigcap_{i\in \Lambda}A_i\p  \bigcap_{i\in \Lambda} B_i.\]
 \end{lem}

\begin{prop} \label{2.17}
Suppose that  \(\ast\) is a continuous t-norm and  \(\va\) is a right continuous t-norm.  Then for any \(f_1, f_2, g\in \mathbf{L_u},\) \(f_1\sqsubseteq  f_2\) implies \( f_1\ast_\va g\sqsubseteq  f_2\ast_\va g\).\end{prop}
\begin{proof}By Proposition \ref{17},   \(f_1^{b}\p f_2^{b}\) for all \(b\in(0,1)\). By Lemma \ref{18},  for all \(b,c\in(0,1)\), \[f_1^{b}\ast g^c\p f_2^{b}\ast g^c.\]  
From Proposition  \ref{.14} and  Lemma \ref{15}, for all \(a\in(0,1)\), 
\[(f_1\ast_\va g)^{\hat{a}}=\bigcup_{b\va c>a}f_1^{b}\ast g^c \p\bigcup_{b\va c>a}f_2^{b}\ast g^c=(f_2\ast_\va g)^{\hat{a}}.\] 
Hence, \( f_1\ast_\va g\sqsubseteq  f_2\ast_\va g\).
\end{proof}
\begin{prop}\cite{hu141}\label{.17}
Let \(\ast\) and \(\va\) be two t-norms. Then  convolution  \(\ast_\va: \mathbf{M}^2\to \mathbf{M}\) is commutative and  \(\bar{1}\) is  the unit.
\end{prop}

As the conclusion of this subsection, we  prove the sufficiency of Theorem \ref{0} as follows.
\begin{prop} 
If   \(\ast\) is a continuous t-norm and  \(\va\) is a right continuous t-norm,  then convolution \(\ast_\va\) is a t-norm on   \((\textbf{L}_u,  \sq)\).
\end{prop}
\begin{proof}
By Proposition  \ref{3.5},   convolution \(\ast_\va\) is closed on \(\textbf{L}_u\).  By Proposition \ref{2.13}, Proposition \ref{2.17} and Proposition \ref{.17}, convolution \(\ast_\va\) is a t-norm on \((\textbf{L}_u,  \sq)\).
\end{proof}

\section{Necessity}

In the subsection, we aim to   prove the necessity of  Theorem \ref{0}.

\begin{prop}\label{0.20}\cite{al06} A binary function  \(T:[0,1]^2\to [0,1]\) is a continuous t-norm if and only if there exists a family of pairwise disjoint  open subintervals  \((a_\alpha, b_\alpha)_{\alpha\in A}\) of  \([0,1]\) and a family of  t-norms  \(T_\alpha: [a_\alpha, b_\alpha]^2\to  [a_\alpha, b_\alpha]\) isomorphic to the product t-norm \(T_P\) or the {\L}uckasiewicz  t-norm \(T_L\)
 such that \(T\) can be represented  by \[  T(x,y)=\begin{cases} T_\alpha(x,y), & (x,y)\in[a_\alpha, b_\alpha]^2,\\ x\wedge y, & \text{otherwise.}
\end{cases}
\] In particular, if  the index set \(A=\varnothing,\)  then \(T=\wedge\).
\end{prop}

\begin{lem}\label{20}
Suppose that  \(T\) is  a continuous t-norm on \([0,1]\) and  \(T\ne \wedge\). Then  there exists  an interval \([a_\al, b_\al]\) such that the restriction of \(T\) on \([a_\al, b_\al]\)  is isomorphic to \(T_p\) or \(T_{L}\), and let \(u,v\in (a_\al, b_\al)\) satisfying \(u\ast v\in (a_\al, b_\al)\). Then  for all \(x,y\in[0,1]\)  satisfying  \(x\ast y=u\ast v\),  \(x=u\) if and only if \(y=v\). Hence, \(x>u\) implies \(y<v\), and \(x<u\) implies  \(y>v.\)
\end{lem}
\begin{proof}
It  follows  from Proposition \ref{0.20} and the fact that both \(T_p\) and \(T_{L}\) are conditionally cancellative.
\end{proof}

 In the following, we prove the necessity of  Theorem \ref{0}.

\begin{prop}
If \(\ast_\va\) is  a t-norm on \((\mathbf{L_u}, \sq)\), then \(\ast\) is a continuous t-norm and  \(\va \) is right continuous t-norm.
\end{prop}
\begin{proof}
By Proposition \ref{8}, we only need to show that t-norm \(\va\) is right continuous, that is, for any \(a,b\in [0,1]^2\),  \(a\va b^+=a\va b\). Obviously, if  \(a\in\{0,1\}\) or  \(b\in\{0,1\}\), then \(a\va b^+=a\va b\). In the following, suppose that \((a,b)\in(0,1)^2\).
To show that \(a\va b^+=a\va b\), we only need to show that there exists a  strictly  decreasing  sequence \(\{y_n\}\) such that \(\inf y_n=b\) and \(\inf   a\va y_n =a\va b\) case by case.

Case 1. \(\ast=\wedge.\)  
Choose arbitrary \(u\in (0,1)\), and define functions \(f,g\in \mathbf{L_u}\)  as  follows. \[ f(x)=\begin{cases} 1, & x=0,\\ a,& x\in (0, u],\\ 0, & \text{otherwise.}\end{cases} \quad g(x)=\begin{cases} \frac{b-1}{u}x+1, & x\in[0,u]\\ 0, & \text{otherwise,}\end{cases}\]   where \(g(0)=1\) and \(g(u)=b\).  Since \(f\) and \(g\) is monotone decreasing, then for any \(z\in[0,1]\), \[(f\wedge_\va g)(z)=\bigvee_{x\wedge y=z}f(x)\va g(y)=(\bigvee_{x=z, y\ge z  }f(x)\va g(y))\vee (\bigvee_{x\ge z, y=z}f(x)\va g(y))=f(z)\va g(z).\]
Hence, \(f\wedge_\va g\) is monotone decreasing, and \[(f\wedge_\va g)(u)=f(u)\va g(u)=a\va b. \]  
Let \(\{x_n\}_{ n\in \mathbb{N}^+}\) be a strictly increasing sequence with limit \(u\), and write \(y_n=g(x_n)\). Obviously,   \(\{y_n\}\)  is  strictly decreasing  with limit \(b\), and
 \[(f\wedge_\va g)(x_n)= f(x_n)  \va g(x_n)=a\va g(x_n)=a\va y_n.\] 
 Since \(f\wedge_\va g\) is monotone decreasing, and \(f\wedge_\va g\in \mathbf{L_u}\) is  upper semicontinuous, then  \[ (f\wedge_\va g)(u-)=(f\wedge_\va g)(u).\]
Since  \(\{x_n\}_{ n\in \mathbb{N}^+}\) is strictly increasing with limit \(u\),   then \[\inf (f\wedge_\va g)(x_n)= (f\wedge_\va g)(u-). \]
From above four equations, we have  \[\inf a\va y_n= \inf (f\wedge_\va g)(x_n)= (f\wedge_\va g)(u-)=(f\wedge_\va g)(u)=a\va b.\] as desired.
 
Case 2. \(\ast\ne \wedge.\)  Since \(\ast\) is a continuous t-norm,  by Proposition \ref{0.20},  there exists an subinterval \([a_\al, b_\al]\) of \([0,1]\) such that the restriction of \(\ast\) on \([a_\al, b_\al]\) is isomorphic to \(T_p\) or \(T_{\L}\). Let \(u,v\in(a_\al,b_\al)\) with \(u\ast v>a_\al\).  Obviously, \(u\ast v<v.\)

Define functions \(f,g\in \mathbf{L_u}\)  as  follows.
 \[f(x)=\begin{cases} 0, & x\in[0, u),\\ a, & x\in[u, 1),\\ 1,& x=1. \end{cases}\quad
 g(x)=\begin{cases}0, & x\in[0,v),\\ \frac{1-b}{1-v}x+\frac{b-v}{1-v},& x\in[v,1],\end{cases}\quad  \] where \(g(v)=b\) and \( g(1)=1\). 
 
 \emph{Firstly}, we show that \((f\ast_\va g)(u\ast v)=a\va b.\) Since \((f\ast_\va g)(u\ast v)\ge f(u)\va g(v)=a\va b,\)  we only need to show that for all \(x,y\in [0,1]^2\) satisfying \(x\ast y=u\ast v\), \(f(x)\va g(y)\le a\va b\). If \(x>u\), by Lemma \ref{20},  \(y<v\), and thus, \(f(x)\va g(y)=0 .\) If \(x=u\),  by Lemma \ref{20}, \(y=v\), and thus, \(f(x)\va g(y)= f(u)\va g(v)=a\va b\). If \(x<u\), then \(f(x)\va g(y)= 0\).  Hence, for all \(x,y\in [0,1]^2\) satisfying \(x\ast y=u\ast v\), \(f(x)\va g(y)\le a\va b\) as desired.

\emph{Next},  let  \(\{x_n\}_{ n\in \mathbb{N}^+}\) be  a strictly decreasing sequence with limit \(v\) satisfying \(u\ast x_n<v\) for each \(n\). The existence of \(\{x_n\}_{ n\in \mathbb{N}^+}\) follows from the fact that \(\ast\) is continuous and \(u\ast v<v\).  
 We claim that \[(f\ast_\va g)(u\ast x_n)=a\va g(x_n).\] Obviously,   
 \[(f\ast_\va g)(u\ast x_n)\ge f(u)\va g(x_n)=a\va g(x_n).\] To show another direction, i.e.   \((f\ast_\va g)(u\ast x_n)\le a\va g(x_n),\) we only need to show that for all \(x,y\in [0,1]^2\) satisfying \(x\ast y=u\ast x_n\), \[f(x)\va g(y)\le a\va g(x_n)\] case by case. 
 
 Case 2.1.  \(0\le x<u\). In this case, \(f(x)\va g(y)=0\va g(y)= 0\);  
 
 Case 2.2.  \(x=u\). By Lemma \ref{20}, \(y=x_n\), and thus,  \(f(x)\va g(y)= f(u)\va g(x_n)=a\va g(x_n).\) 
 
 Case 2.3.  \(u<x<1\). By Lemma \ref{20}, \(y<x_n\). Since  \(g\) is monotone  increasing,  \[f(x)\va g(y)\le f(x)\va g(x_n)=a\va g(x_n).\]
 
Case 2.4. \(x=1\). Since \(y=u\ast x_n<v\),  then \(g(y)=0\) and  \(f(x)\va g(y)=0.\)

\emph{Last}, let \(y_n=g(x_n),\) then \(\inf y_n=\inf g(x_n)=g (\inf x_n)=g(v)= b\). Since the sequence \[\{(f\ast_\va g)(u\ast x_n)\}=\{a\va g(x_n)\}\] is monotone decreasing, by Lemma \ref{1},  \[\inf (f\ast_\va g)(u\ast x_n)=(f\ast_\va g)(\lim_{n\ra \infty }u\ast x_n)=(f\ast_\va g)(u\ast v).\]
Hence,  \[\inf a\va y_n=\inf a\va g(x_n)=\inf (f\ast_\va g)(u\ast x_n)=(f\ast_\va g)(u\ast v)=a\va b\] as desired.
 \end{proof}

  \section{Conclusion}
 In this study, we have investigated the construction of t-norms on \((\mathbf{L_u},\sq)\). The main results are as follows:
 \begin{itemize}
 \item[1)] The convolution \(\ast_\va\) is  a   t-norm on \((\mathbf{L}_\mathbf{u},\sq)\) if and only if it is a \(t_r\)-norm on  \((\mathbf{L}_\mathbf{u},\sq)\), see Proposition \ref{0.1}.  
 \item[2)] The convolution \(\ast_\va\) is a t-norm on \((\mathbf{L_u}, \sq)\)  if and only if \(\ast\) is a continuous t-norm  and \(\va\) is a right continuous t-norm, see Theorem \ref{0}.
  \end{itemize}

\vfill
\end{document}